\documentclass[12pt]{article}
\usepackage{amssymb,amsthm,amsmath,latexsym}
\newtheorem{thm}{Theorem}

\newtheorem{lem}[thm]{Lemma}
\newtheorem{cor}[thm]{Corollary}
\theoremstyle{remark}

\newcommand{\ZZ}{\mathbb{Z}}
\newcommand{\CC}{\mathbb{C}}
\newcommand{\cD}{\mathcal{D}}
\newcommand{\ww}{\omega}
\newcommand{\vv}{\overline{\omega}}
\newcommand{\nexteq}{\displaybreak[0]\\ &=}
\DeclareMathOperator{\Aut}{Aut}

\DeclareMathOperator{\Nei}{Nei}
\DeclareMathOperator{\GL}{GL}
\DeclareMathOperator{\diag}{diag}

\begin{document}

\title{Classification of Generalized Hadamard Matrices $H(6,3)$
and Quaternary Hermitian Self-Dual Codes of Length 18}

\author{
Masaaki Harada\thanks{
Department of Mathematical Sciences,
Yamagata University,
Yamagata 990--8560, Japan, and
PRESTO, Japan Science and Technology Agency, Kawaguchi,
Saitama 332--0012, Japan},
Clement Lam\thanks{Department of Computer Science,
Concordia University, Montreal, QC, Canada, H3G 1M8},
 Akihiro Munemasa\thanks{
Graduate School of Information Sciences,
Tohoku University,
Sendai 980--8579, Japan},\\
and\\
Vladimir D. Tonchev\thanks{ Mathematical Sciences, Michigan Technological University,
Houghton, MI 49931, USA}
}

\maketitle

\begin{abstract}

All generalized Hadamard matrices of order 18
over a group of order 3, $H(6,3)$, are enumerated in two different ways:
once, as class regular 
symmetric $(6,3)$-nets, or symmetric transversal designs on 
54 points and 54 blocks with a group of order 3 acting 
semi-regularly on points and blocks,
and secondly, as collections of full weight vectors
in quaternary Hermitian self-dual codes of length 18. 
The second enumeration is based on the classification of 
Hermitian self-dual $[18,9]$ codes over $GF(4)$,
completed in this paper.
It is shown that
up to monomial equivalence, there are 85 generalized Hadamard matrices
$H(6,3)$,
and 245 inequivalent  Hermitian self-dual codes of length 18
over $GF(4)$.
 
\end{abstract}

\section{Introduction}

A {\em generalized Hadamard matrix} $H(\mu,g)=(h_{ij})$ 
of order $n=g\mu$ over a
multiplicative group $G$ of order $g$
is a $g\mu \times g\mu$ matrix with entries from $G$ with the property that
for every $i$, $j$, $1 \le i< j \le g\mu$, each of the multi-sets
$\{ h_{is}h_{js}^{-1} \ | \ 1\le s \le g\mu \}$
contains every element of $G$
exactly $\mu$ times.
It is known  \cite[Theorem 2.2]{jug}
that if $G$ is abelian then
$H(\mu,g)^T$ is also a generalized Hadamard matrix,
where $H(\mu,g)^T$ denotes the transpose of $H(\mu,g)$
(see also \cite[Theorem 4.11]{Brock}).
This result does not generalize to non-abelian groups,
as shown by Craigen and de Launey \cite{CL09}.

If $G$ is an additive group and the products $h_{is}h_{js}^{-1}$
are replaced by differences $h_{is}-h_{js}$, the resulting matrices
are known as {\em difference matrices} \cite{BJL}, or 
{\em difference schemes}  \cite{HSS}.
A generalized Hadamard matrix over the multiplicative group 
of order two, $G=\{ 1, -1 \}$, is an ordinary Hadamard matrix.

Permuting rows or columns, as well as multiplying rows or columns
of a given generalized Hadamard matrix $H$ over a group $G$
with group elements changes $H$ into another generalized Hadamard
matrix.
Two generalized Hadamard matrices $H'$, $H''$ of order $n$ over a group $G$
are called {\em monomially equivalent} if $H''=PH'Q$ for some
monomial matrices $P$, $Q$ of order $n$ with nonzero entries from $G$.

All generalized Hadamard matrices over a group of order 2, that is, ordinary Hadamard
matrices, have been classified up to (monomial) 
equivalence for all orders up to $n=28$
\cite{Ki}, and all generalized Hadamard matrices over a group of order 4 
(cyclic or elementary abelian) have 
been classified up to monomial equivalence
for all orders up to $n=16$ \cite{HLT} (see also \cite{GM}).

We consider generalized Hadamard matrices over a group of order 3
in this paper.
It is easy to verify that generalized Hadamard matrices
$H(1,3)$ of order 3, and $H(2,3)$ of order 6, exist and are unique up to 
monomial equivalence.
There are 
two matrices $H(3,3)$ of order 9 \cite{MT}, and one $H(4,3)$ of order 12 
up to monomial equivalence \cite{Su}.
It is known \cite[Theorem 6.65]{HSS} that an $H(5,3)$ of order 15 does
not exist.
Up to monomial equivalence,
at least 11 $H(6,3)$ of order 18 were previously known \cite{AOS}. 
 
In this paper, we enumerate all generalized Hadamard matrices
$H(6,3)$ of order 18, up to monomial equivalence.
We present two different 
enumerations, one based on combinatorial designs known as symmetric nets
or transversal designs (Section \ref{nets}), and a second
enumeration based on the classification of Hermitian self-dual codes 
of length 18 over $GF(4)$ completed in Section \ref{Sec:C}.

\section{Symmetric nets, transversal designs and generalized Hadamard matrices $H(6,3)$}
\label{nets}

A {\em symmetric \ $(\mu,g)$-net} is a 
1-$(g^2{\mu}, g{\mu}, g{\mu})$
design $\cal{D}$ such that both $\cal{D}$ and its dual design 
$\cal{D^*}$ are affine
resolvable \cite{BJL}:  the $g^2{\mu}$ points of $\cal{D}$ 
are partitioned into $g{\mu}$  parallel classes, or {\em groups},
each containing $g$ points, so that any two points which belong to the 
same class do not occur together in any block, while any two points 
which belong to different classes occur together in exactly $\mu$ blocks. 
Similarly, the blocks are partitioned into $g\mu$ parallel classes, each 
consisting of $g$ pairwise disjoint blocks, and any two blocks which belong 
to different parallel classes share exactly $\mu$ points. 
A symmetric $(\mu,g)$-net is also known as a 
{\em symmetric transversal design}, 
and denoted by $STD_{\mu}(g)$, or $TD_{\mu}(g\mu,g)$ \cite{BJL}, 
or $STD_{\mu}[g\mu;g]$ \cite{Su}.
A symmetric $(\mu,g)$-net is {\em class-regular} if it admits a group of
automorphisms $G$ of order $g$ (called group of {\em bitranslations})
 that acts transitively (and hence regularly)
on every point and block parallel class.

Every generalized Hadamard matrix $H(\mu,g)$ over a group $G$ of 
order $g$ determines
a class-regular symmetric $(\mu,g)$-net with a group of bitranslations
isomorphic to $G$, and conversely, every class-regular $(\mu,g)$-net
with a group of bitranslations $G$ gives rise to a generalized Hadamard matrix
$H(\mu,g)$ \cite{BJL}.
The $g^2\mu \times g^2\mu$ $(0,1)$-incidence matrix of
a class-regular symmetric $(\mu,g)$-net  is obtained from a given 
generalized Hadamard matrix $H(\mu,g)=(h_{ij})$ over a group $G$ 
of order $g$ by replacing each entry $h_{ij}$ of $H(\mu,g)$ with
a $g \times g$ permutation matrix representing $h_{ij}\in G$.
This correspondence relates the task of enumerating generalized 
Hadamard matrices over a group of order $g$ to the enumeration
of 1-$(g^2\mu,g\mu,g\mu)$ designs with incidence matrices
composed of $g \times g$ permutation submatrices. 
This approach was used in \cite{HLT}
for the enumeration of all nonisomorphic class-regular symmetric $(4,4)$-nets
over a group of order 4 and generalized Hadamard
matrices $H(4,4)$. In this paper, we use the same approach to enumerate
all pairwise nonisomorphic class-regular $(6,3)$-nets, or equivalently,
symmetric transversal designs $STD_{6}(3)$ with a group of order 3 acting
semiregularly on point and block parallel classes, and consequently,
all generalized Hadamard matrices $H(6,3)$. As in \cite{HLT}, the 
block design exploration package BDX \cite{BDX}, developed by Larry Thiel, 
was used for the enumeration.

The results of this computation can be formulated as follows. 
\begin{thm}
\label{t1}
Up to isomorphism, there are exactly $53$ class-regular symmetric
$(6,3)$-nets, or equivalently, $53$ symmetric  transversal designs 
$STD_{6}(3)$ with a group of order $3$ acting
semiregularly on point and block parallel classes.
\end{thm}

The information about 
the 53 $(6,3)$-nets $\cD_i$ $(i=1,2,\ldots,53)$
are listed in Table \ref{Tab:Net}.
In the table, $\#\Aut$ 
gives the size of the automorphism group of $\cD_i$.
The column $\cD_i^*$ gives the number $j$, where
$\cD_i^*$ is isomorphic to $\cD_j$.
Incidence matrices of the 53 $(6,3)$-nets are available at
\begin{verbatim}
 www.math.mtu.edu/~tonchev/sol.txt.
\end{verbatim}

We note that 20 nonisomorphic $STD_{6}(3)$ were found by Akiyama, Ogawa, and
Suetake \cite{AOS}.
These twenty $STD_{6}(3)$ are denoted by 
$\cD(H_i)$ $(i=1,2,\ldots,11)$ and $\cD(H_i)^d$ $(i=1,\ldots,5,7,8,9,10)$
in \cite[Theorem 7.3]{AOS}.
When $\cD_i$ in Table \ref{Tab:Net} is isomorphic to
one of the twenty $STD_{6}(3)$ in \cite{AOS},
we list the  $STD_{6}(3)$ in the column $D_{AOS}$ of the table.

\begin{table}[thbp]
\caption{Class-regular symmetric $(6,3)$-nets and $H(6,3)$'s}
\label{Tab:Net}
\begin{center}
{\footnotesize
\begin{tabular}{|cccc|c||cccc|c|}
\hline
$\cD_i$ & $\#\Aut$ & $\cD_i^*$ &$D_{AOS}$& $\overline{H(\cD_i)}$ &
$\cD_i$ & $\#\Aut$ & $\cD_i^*$ &$D_{AOS}$& $\overline{H(\cD_i)}$ \\
\hline 
1 &      96 &     1 &&    yes &28&     162 &    37 &$\cD(H_{1})^d$&    yes \\
2 &     432 &    43 &&    yes &29&      54 &    22 &&     no \\
3 &     864 &     5 &&    yes &30&      54 &    26 &&     no \\
4 &   38880 &     4 &$\cD(H_{11})$&    yes &31&     432 &    17 &&     no \\
5 &     864 &     3 &&    yes &32&      48 &    15 &&     no \\
6 &    1296 &    19 &&    yes &33&      54 &    27 &&    yes \\
7 &    3240 &    49 &$\cD(H_{10})$&     no &34&     162 &    53 &$\cD(H_{2})$&     no \\
8 &     144 &    46 &&     no &35&     162 &    50 &$\cD(H_{4})$&     no \\
9 &     324 &    44 &$\cD(H_{5})$&     no &36&     162 &    51 &$\cD(H_{3})$&     no \\
10&    1296 &    52 &$\cD(H_{7})$&     no &37&     162 &    28 &$\cD(H_{1})$&    yes \\
11&     180 &    45 &&     no &38&    1944 &    14 &$\cD(H_{9})$&    yes \\
12&    1296 &    42 &$\cD(H_{8})$&    yes &39&     972 &    39 &$\cD(H_{6})$&    yes \\
13&     216 &    20 &&    yes &40&     216 &    21 &&     no \\
14&    1944 &    38 &$\cD(H_{9})^d$&    yes &41&     216 &    16 &&     no \\
15&      48 &    32 &&     no &42&    1296 &    12 &$\cD(H_{8})^d$&    yes \\
16&     216 &    41 &&     no &43&     432 &     2 &&    yes \\
17&     432 &    31 &&     no &44&     324 &     9 &$\cD(H_{5})^d$&     no \\
18&    2160 &    23 &&    yes &45&     180 &    11 &&     no \\
19&    1296 &     6 &&    yes &46&     144 &     8 &&     no \\
20&     216 &    13 &&    yes &47&     108 &    24 &&     no \\
21&     216 &    40 &&     no &48&    1080 &    25 &&     no \\
22&      54 &    29 &&     no &49&    3240 &     7 &$\cD(H_{10})^d$&     no \\
23&    2160 &    18 &&    yes &50&     162 &    35 &$\cD(H_{4})^d$&     no \\
24&     108 &    47 &&     no &51&     162 &    36 &$\cD(H_{3})^d$&     no \\
25&    1080 &    48 &&     no &52&    1296 &    10 &$\cD(H_{7})^d$&     no \\
26&      54 &    30 &&     no &53&     162 &    34 &$\cD(H_{2})^d$&     no \\
27&      54 &    33 &&    yes & &&&&\\
\hline
   \end{tabular}
}
\end{center}
\end{table}

Any generalized Hadamard matrix $H(6,3)$
over the group $G =\{1,\ww , \ww^2  \mid \ww^3=1 \}$ corresponds 
to the $54 \times 54$ $(0,1)$-incidence matrix of a 
class-regular symmetric $(6,3)$-net  obtained by 
replacing $1,\ww$ and $\ww^2$ with
$3 \times 3$ permutation matrices $I,M_3$ and $M_3^2$,
respectively, where $I$ is the identity matrix and 
\[
M_3=
\left(\begin{array}{ccc}
0 &  1 &  0\\
0 &  0 &  1\\
1 &  0 &  0
\end{array}\right).
\]
We note that permuting rows or columns in $H(6,3)$
corresponds to permuting parallel classes of points or blocks in the 
related symmetric
net, while multiplying a row or column of $H(6,3)$ with an 
element $\alpha$ of $G$,
corresponds to a cyclic shift (if $\alpha=\ww$)
or a double cyclic shift (if $\alpha=\ww^2$) of the three points or 
blocks of the corresponding parallel class in the related symmetric 
$(6,3)$-net. 
Thus, monomially equivalent
generalized Hadamard matrices $H(6,3)$ correspond to isomorphic 
symmetric $(6,3)$-nets. 

The inverse operation of replacing every element $h_{ij}$ of a
generalized Hadamard matrix by its inverse $h_{ij}^{-1}$ also
preserves the property of being a generalized Hadamard matrix.
That is, a generalized Hadamard matrix is also obtained by 
replacing $I,M_3$ and $M_3^2$ with $1,\ww^2$ and $\ww$, respectively.
However, this
is not considered a monomial equivalence operation. As a symmetric net, this
inverse operation corresponds to replacing $M_3$ by $M_3^2$ and vice
versa. The inverse operation is achievable by simulataneously
interchanging rows 2 and 3 and columns 2 and 3 of the matrices $I$,
$M_3$ and $M_3^2$. Thus, by simulataneous interchanging points 2 and 3 and
blocks 2 and 3 of every parallel class of points and blocks, the inverse
operator is an isomorphism operation of symmetric nets.
Since the definition of isomorphic symmetric nets and monomially
equivalent generalized Hadamard matrices differs only in the extra
inverse operation, at most two generalized
Hadamard matrices which are not monomially equivalent 
can arise from a symmetric net.
We note that for generalized Hadamard matrices over a cyclic group of 
order $q$, replacing every entry by its $i$-th power, where gcd$(i,q)=1$, 
may give a generalized Hadamard matrix which is not monomially 
equivalent to 
the original; however, their corresponding symmetric nets are isomorphic.

In order to find the number of generalized Hadamard matrices
which are not monomially equivalent,
we first convert the $53$ nonisomorphic
symmetric nets into their corresponding $53$ generalized Hadamard
matrices. We then create a list of $53$ extra matrices by 
applying the inverse operation. Amongst this list of 106 matrices, we found
$85$ generalized Hadamard matrices $H(6,3)$ up to monomial equivalence.
As expected, the remaining $21$ matrices are monomially equivalent to their
``parent'' before the inverse operation.

\begin{cor}\label{Cor:H63}
Up to monomial equivalence,
there are exactly $85$ generalized Hadamard matrices $H(6,3)$.
\end{cor}

In Table \ref{Tab:Net},
the column $\overline{H(\cD_i)}$ states whether the corresponding 
generalized Hadamard
matrix $H(\cD_i)$ is monomially equivalent to the generalized 
Hadamard matrix $\overline{H(\cD_i)}$
obtained by replacing all entries by their inverse.
Thus, the set
$\{H(\cD_i),\overline{H(\cD_j)} \mid i \in \Delta, j \in 
\Delta \setminus \Gamma\}$
gives the $85$ generalized Hadamard matrices, 
where $\Delta=\{1,2,\ldots,53\}$ and
\[
\Gamma=\{
1, 2, 3, 4, 5, 6, 12, 13, 14, 
18, 19, 20, 23, 27, 28, 33, 37, 38, 39, 42, 43
\}.
\]

Concerning the next order, $n=21$,
several examples of $STD_{7}(3)$ and  $H(7,3)$ are known
\cite{AOS}, \cite{Su1}.  Some $STD_{7}(3)$'s and $H(7,3)$'s were used
in \cite{T07} as building blocks for the construction of an infinite
class of quasi-residual 2-designs.
An estimate based on preliminary computations with BDX suggests that it would take
500 CPU years to enumerate all $STD_{7}(3)$'s using one computer, or
about a year of CPU if a network of 500 computers is employed.

\section{Elementary divisors of generalized Hadamard matrices
and Hermitian self-dual codes}

Let $GF(4)=\{ 0,1,\ww , \vv  \}$ be the finite field of order
four, where $\vv=  \omega^2 = \omega +1$.
Codes over $GF(4)$ are often called quaternary.
The {\em Hermitian inner product} of vectors
$x=(x_1,\ldots,x_n), y=(y_1,\ldots,y_n)\in GF(4)^n$
is defined as
\begin{equation}
\label{herm}
x \cdot y = \sum_{i=1}^{n} x_i {y_i}^2.
\end{equation}
The \textit{Hermitian dual code} $C^{\perp}$ of a code 
$C$ of length $n$ is defined as
$
C^{\perp}=
\{x \in GF(4)^n \mid x \cdot c = 0 \text{ for all } c \in C\}.
$
A code $C$ is  called
\textit{Hermitian self-orthogonal}
if $C\subseteq C^{\perp}$,
and \textit{Hermitian self-dual} if $C = C^{\perp}$. 
In this section, we show that
the rows of any generalized Hadamard matrix $H(6,3)$  
span a Hermitian self-dual code of length 18
and minimum weight $d\ge 4$
(Theorem \ref{thm:sd}).
A consequence of this result is that all $H(6,3)$'s can be found
as collections of vectors of full weight in
Hermitian self-dual codes over $GF(4)$.
This motivates us to classify all such codes
as the second approach of the enumeration of all $H(6,3)$'s.

Let $R$ be a unique factorization domain, and let $p$ be a prime element
of $R$. For a nonzero element $a\in R$, we denote by $\nu_p(a)$
the largest non-negative integer $e$ such that $p^e$ divides $a$.

\begin{lem}\label{lem:gcd}
Let $R$ be a unique factorization domain.
Suppose that the nonzero elements $a,b,c,d\in R$ satisfy 
$ab=cd$ and $\gcd(a,b)=1$. Then
\[
\gcd(a,c)\gcd(a,d)=a.
\]
\end{lem}
\begin{proof}
Let $p$ be a prime element of $R$ dividing $a$. Then $p$
does not divide $b$, hence
\[
\nu_p(a)=\nu_p(ab)=\nu_p(c)+\nu_p(d)\geq\max\{\nu_p(c),\nu_p(d)\}.
\]
Thus
\begin{align*}
\nu_p(\gcd(a,c))=\min\{\nu_p(a),\nu_p(c)\}&=\nu_p(c),\\
\nu_p(\gcd(a,d))=\min\{\nu_p(a),\nu_p(d)\}&=\nu_p(d),
\end{align*}
and hence
$\nu_p(a)=\nu_p(\gcd(a,c)\gcd(a,d))$. Since $p$ is arbitrary,
we obtain the assertion.
\end{proof}

Let $\omega=\frac{-1+\sqrt{-3}}{2} \in \CC$,
where $\CC$ denotes the complex number field.
It is well known that
$\ZZ[\omega]$ is a principal ideal domain. Thus we can consider
elementary divisors of a matrix over $\ZZ[\omega]$. Also,
$\ZZ[\omega]$ is a unique factorization domain, and $2$ is a prime
element. We note that $\ZZ[\omega]/2\ZZ[\omega]\cong GF(4)$.

\begin{lem}\label{lem:divH}
Let $H$ be an $n\times n$ matrix with entries in $\{1,\omega,\omega^2\}$,
satisfying $H\overline{H}^T=nI$, 
where $\overline{H}$ denotes the complex conjugation.  Let 
$d_1|d_2|\cdots|d_n$ be the elementary divisors of $H$
over the ring $\ZZ[\omega]$.
Then $d_i\overline{d_{n+1-i}}/n$ is a unit in $\ZZ[\omega]$
for all $i=1,\dots,n$.
\end{lem}
\begin{proof}
Take $P,Q\in \GL(n,\ZZ[\omega])$ so that $PHQ=\diag(d_1,\dots,d_n)$.
Since $H\overline{H}^T=nI$, we have
\begin{align*}
\overline{Q}^{-1}H^T\overline{P}^{-1}
&=
n\overline{Q}^{-1}\overline{H}^{-1}\overline{P}^{-1}
\nexteq
n\overline{PHQ}^{-1}
\nexteq
\diag(n/\overline{d_1},n/\overline{d_{2}},\dots,n/\overline{d_n}).
\end{align*}
This implies that 
$n/\overline{d_n},n/\overline{d_{n-1}},\dots,n/\overline{d_1}$
are also the elementary divisors of $H$. 
It follows from the uniqueness of the elementary divisors that
$d_i\overline{d_{n+i-i}}/n$ is a unit in $\ZZ[\omega]$ for
all $i=1,\dots,n$.
\end{proof}


\begin{thm}\label{thm:sd}
Under the same assumptions as in Lemma~\ref{lem:divH}, assume
further that $n\equiv2\pmod4$.
Then the rows of $H$ span 
a Hermitian self-dual code over $\ZZ[\omega]/2\ZZ[\omega]\cong GF(4)$.
This Hermitian self-dual code has minimum weight at least $4$.
\end{thm}
\begin{proof}
Let $C$ be the code over $\ZZ[\omega]/2\ZZ[\omega]$
spanned by the row vectors
of $H$. 
Since $H\overline{H}^T\equiv0\pmod{2\ZZ[\omega]}$, the code $C$ is Hermitian
self-orthogonal (see also \cite[Lemma 2]{T09}).
 Let $d_1|d_2|\cdots|d_n$ be the elementary divisors of $H$. 
Then
\begin{align*}
|C|&=|(\ZZ[\omega]/2\ZZ[\omega])^nH|
\nexteq
|(\ZZ[\omega]/2\ZZ[\omega])^n\diag(d_1,\dots,d_n)|
\nexteq
\prod_{i=1}^{n}|\gcd(2,d_i)\ZZ[\omega]/2\ZZ[\omega]|
\nexteq
\prod_{i=1}^{n}
\frac{|\ZZ[\omega]/2\ZZ[\omega]|}{|\ZZ[\omega]/
\gcd(2,d_i)\ZZ[\omega]|}
\nexteq
\prod_{i=1}^{n}
\frac{4}{|\gcd(2,d_i)|^2}
\nexteq
\prod_{i=1}^{n/2}
\frac{4}{|\gcd(2,d_i)|^2}
\prod_{i=n/2+1}^{n}
\frac{4}{|\gcd(2,d_i)|^2}
\nexteq
\prod_{i=1}^{n/2}
\frac{4}{|\gcd(2,d_i)|^2}
\prod_{i=n/2+1}^{n}
\frac{4}{|\gcd(2,n/\overline{d_{n+1-i}})|^2}
\qquad \text{(by Lemma~\ref{lem:divH})}
\nexteq
\prod_{i=1}^{n/2}
\frac{4}{|\gcd(2,d_i)|^2}
\prod_{i=n/2+1}^{n}
\frac{4}{|\gcd(2,n/d_{n+1-i})|^2}
\nexteq
\prod_{i=1}^{n/2}
\frac{4}{|\gcd(2,d_i)|^2}
\prod_{i=1}^{n/2}
\frac{4}{|\gcd(2,n/d_i)|^2}
\nexteq
\prod_{i=1}^{n/2}
\frac{16}{|\gcd(2,d_i)\gcd(2,n/d_i)|^2}
\nexteq
4^{n/2}.
\hspace{4cm} \text{(by Lemma~\ref{lem:gcd} since $n \equiv 2 \pmod 4$)}
\end{align*}
Thus, the dimension $\dim C$ is $n/2$ and $C$ is self-dual.


If the dual code $C^\perp$ had minimum weight $2$, 
then there exist two columns of $H$, one of which is 
a multiple by $1, \ww$, or $\vv$ of the other, in $GF(4)$. 
But this implies that there exists a column of $H$ which is
a multiple by $1, \ww$, or $\vv$ in $\CC$.
This is impossible since $H$ is nonsingular.
Hence the dual code $C^\perp$ has minimum weight at least $3$.
Since $C$ is self-dual and even, $C$ has minimum weight at least $4$.
\end{proof}


\section{The classification of quaternary self-dual $[18,9]$ codes }
\label{Sec:C}

Two codes $C$ and $C'$ over $GF(4)$
are \textit{equivalent} if there is a monomial
matrix $M$ over $GF(4)$ such that $C' =C M =\{c M \mid c \in C \}$.
A monomial matrix which maps $C$ to itself is called an {\em automorphism} 
of $C$ and  the set of all automorphisms of $C$ forms the 
automorphism group $\Aut(C)$ of $C$.
The number of distinct Hermitian self-dual codes of 
length $n$ is given  \cite{MOSW} by the formula: 
\begin{equation}
\label{num}
N(n)=\prod_{i=0}^{n/2-1}(2^{2i+1}+1).
\end{equation}

It was shown in  \cite{MOSW} that
the minimum weight $d$ of a Hermitian self-dual code of
length $n$ is bounded by
$d \leq 2 \lfloor n/6 \rfloor +2$.
A Hermitian self-dual code of length $n$ and minimum weight 
$d=2 \lfloor n/6 \rfloor +2$ is called \textit{extremal}. 
The classification of all Hermitian self-dual codes over $GF(4)$ up 
to equivalence 
of length $n\le14$ was completed by MacWilliams, Odlyzko, Sloane and Ward
\cite{MOSW}, and the Hermitian self-dual codes of length 16 were classified 
by Conway, Pless and Sloane  \cite{CPS}.
For example, 
there are $55$ inequivalent Hermitian self-dual codes of length $16$.
For the next two lengths, $18$ and $20$, only partial classification was
previously known, namely, the extremal 
Hermitian self-dual $[18,9,8]$ and $[20,10,8]$ 
codes were enumerated in \cite{Huffman} and
Hermitian self-dual $[18,9,6]$ codes were enumerated 
in \cite{BO06}
under the weak equivalence defined at the end of this subsection.

We first consider decomposable Hermitian self-dual codes.
By \cite[Theorem 28]{MOSW}, any Hermitian self-dual code
with minimum weight $2$ is decomposable as $C_2\oplus C_{16}$,
where $C_2$ is the unique Hermitian self-dual code of length $2$
and $C_{16}$ is some Hermitian self-dual code of length $16$.
Hence, there are $55$ inequivalent Hermitian self-dual codes
with minimum weight $2$ \cite{CPS}.
In the notation of Table \ref{Tab:C}, the following codes
are decomposable Hermitian self-dual codes with minimum
weight $4$:
\[
E_8\oplus E_{10},
E_8\oplus B_{10},
E_6\oplus E_{12},
E_6\oplus C_{12},
E_6\oplus D_{12},
E_6\oplus F_{12},
E_6\oplus 2E_6,
\]
and there is no decomposable Hermitian self-dual code with minimum
weight $d \ge 6$.
In Table \ref{Tab:18}, the number $\#_{\text{dec}}$
of inequivalent decomposable Hermitian self-dual codes with minimum
weight $d$ is given for each admissible value of $d$.

\begin{table}[thb]
\caption{Hermitian self-dual codes of length $18$}
\label{Tab:18}
\begin{center}
{\small
\begin{tabular}{|c|cccc|c|}
\hline
                  & $d=2$ & $d=4$ & $d=6$ & $d=8$ & Total \\
\hline
$\#_{\text{dec}}$ &  55   &    7  &   0   &    0  &   62  \\
$\#_{\text{indec}}$ &   0   &  152  &  30   &    1  &  183  \\
\hline
Total             &  55   &  159  &  30   &    1  &  245  \\
\hline
   \end{tabular}
}
\end{center}
\end{table}

We now consider indecomposable Hermitian self-dual codes.
Two self-dual codes $C$ and $C'$ of length $n$
are called {\em neighbors} if the dimension of their intersection 
is $n/2-1$.
An extremal Hermitian self-dual code $S_{18}$ of length $18$
was given in \cite{MOSW} and it is generated by
\[
(1,\ww,\vv,\ww,\ww,\ww,\vv,\vv,\vv,\vv,\vv,\vv,\ww,\ww,\ww,\vv,\ww)~1
\]
where the parentheses indicate 
that all cyclic shifts are to be used.
Let $\Nei(C)$ denote the set of inequivalent Hermitian self-dual
neighbors with minimum weight $d \ge 4$ of $C$.
We found that the set $\Nei(S_{18})$ consists of 
$35$ inequivalent Hermitian self-dual codes,
one of which is equivalent to $S_{18}$,
$17$ codes have minimum weight $6$, and 
$17$ codes have minimum weight $4$.
Within the set of codes
\[
\{S_{18}\} \cup
\Nei(S_{18}) \cup
{\cal N} \cup \Big(\bigcup_{C \in {\cal N}} \Nei(C) \Big),
\]
where ${\cal N}=\cup_{C \in \Nei(S_{18})} \Nei(C)$,
we found a set $\mathcal{C}_{18}$ of $190$
inequivalent Hermitian self-dual codes $C_1,\ldots,C_{190}$
with minimum weight $d \ge 4$ 
satisfying
\begin{equation}
\label{mass}
\sum_{C \in \mathcal{C}_{18} \cup \mathcal{D}_{18}} 
\frac{3^{18} \cdot 18!}{\#\Aut(C)} = 4251538544610908358733563 =N(18),
\end{equation}
where $\mathcal{D}_{18}$ denotes the set of the $55$ inequivalent
Hermitian self-dual codes of length $18$ and 
minimum weight $2$.
The orders of the automorphism groups of the $245$ codes
in $\mathcal{C}_{18}\cup \mathcal{D}_{18}$
are listed in Table \ref{Tab:Aut}.
The mass formula (\ref{mass}) shows that 
the set $\mathcal{C}_{18} \cup \mathcal{D}_{18}$ of codes 
contains representatives of all equivalence classes of
Hermitian self-dual codes  
of length $18$. Thus,  the classification is complete,
and Theorem \ref{thm:18} holds.

\begin{thm}\label{thm:18}
There are $245$ inequivalent Hermitian self-dual codes of length $18$.
Of these,  one is extremal (minimum weight $8$), 
$30$ codes have minimum weight $6$,
$159$ codes have minimum weight $4$, and $55$ codes have minimum weight $2$.
\end{thm}
The software package {\sc Magma}~\cite{Magma} was used in the computations.
Generator matrices of all Hermitian self-dual codes of length $18$
can be obtained electronically from 
\begin{verbatim}
 www.math.is.tohoku.ac.jp/~munemasa/selfdualcodes.htm.
\end{verbatim}

\begin{table}[thb]
\caption{Orders of the automorphism groups}
\label{Tab:Aut}
\begin{center}
{\footnotesize
\begin{tabular}{|c|l|}
\hline
$d$ & \multicolumn{1}{c|}{$\#\Aut(C)$} \\
\hline
2
&
864, 864, 1152, 1728, 2160, 2304, 2592, 6048, 6912, 6912, 10368, 13824, 13824, 
\\ &
17280, 20736, 41472, 82944, 82944, 82944, 82944, 103680, 110592, 124416, 235872, 
\\ &
248832, 311040, 331776, 497664, 580608, 829440, 995328, 995328, 1327104, 
\\ &
2073600, 2177280, 2488320, 3110400, 4478976, 12192768, 13436928, 18662400, 
\\ &
37324800, 39191040, 69672960, 89579520, 92897280, 139968000, 179159040, 
\\ &
195084288, 313528320, 671846400, 3023308800, 3762339840, 36279705600, 
\\ &
3656994324480
\\
\hline
4
&
24, 24, 24, 24, 24, 24, 24, 36, 48, 48, 48, 48, 72, 72, 72, 72, 72, 72, 96, 
96, 96, 96, 
\\ &
96, 96, 96, 144, 144, 144, 144, 144, 192, 192, 192, 192, 192, 192, 192, 192, 
192, 
\\ &
288, 288, 288, 288, 288, 288, 288, 288, 288, 384, 384, 384, 384, 
384, 384, 432, 
\\ &
504, 576, 576, 576, 768, 768, 864, 864, 1152, 1152, 1152, 1152, 
1152, 1152, 1152, 
\\ &
1152, 1152, 1152, 1152, 1152, 1536, 1728, 2304, 2304, 2304, 
2304, 2304, 3072, 
\\ &
3456, 3456, 4608, 4608, 4608, 5760, 6144, 6912, 6912, 6912, 
6912, 6912, 6912, 
\\ &
6912, 9216, 10368, 10368, 12960, 13824, 13824, 13824, 13824, 
13824, 13824, 
\\ &
14040, 17280, 17280, 18432, 18432, 18432, 20736, 27648, 27648, 
34560, 48384, 
\\ &
51840, 55296, 55296, 55296, 55296, 62208, 69120, 82944, 82944, 
103680, 124416, 
\\ &
138240, 138240, 145152, 207360, 207360, 221184, 221184, 248832, 
248832, 248832, 
\\ &
414720, 518400, 552960, 725760, 967680, 1105920, 1658880, 2032128, 
3110400, 
\\ &
3732480, 4147200, 4147200, 11197440, 11664000, 23224320, 32659200, 
74649600, 
\\ &
87091200, 278691840, 7558272000
\\
\hline
6
&
6, 12, 12, 12, 12, 12, 18, 24, 24, 27, 36, 36, 36, 36, 36, 54, 54, 72, 96, 
180, 180, 
\\ &
216, 216, 288, 648, 1080, 1152, 1296, 2916, 23328
\\
\hline
8 & 24480 \\
\hline
   \end{tabular}
}
\end{center}
\end{table}

In Table \ref{Tab:18}, the number $\#_{\text{indec}}$
of indecomposable Hermitian self-dual codes with minimum
weight $d$ is given.
In Table \ref{Tab:C}, the number $\#$ of inequivalent
Hermitian self-dual codes of length $n$ is given
along with references.
The largest minimum weight $d_{\text{max}}$ among
Hermitian self-dual codes of length $n$ and
the number $\#_{\text{max}}$ of inequivalent
Hermitian self-dual codes with minimum weight $d_{\text{max}}$
are also listed along with references.

\begin{table}[thb]
\caption{Hermitian self-dual codes}
\label{Tab:C}
\begin{center}
{\small
\begin{tabular}{|c|cl|ccl|}
\hline
$n$ & $\#$ & \multicolumn{1}{c|}{References} &
$d_{\text{max}}$ & $\#_{\text{max}}$ & \multicolumn{1}{c|}{References}\\
\hline
 2 &   1 & \cite{MOSW}& 2 & 1 & $C_2$ in \cite{MOSW} \\
 4 &   1 & \cite{MOSW}& 2 & 1 & $2C_2$ in \cite{MOSW} \\
 6 &   2 & \cite{MOSW}& 4 & 1 & $E_6$ in \cite{MOSW} \\
 8 &   3 & \cite{MOSW}& 4 & 1 & $E_8$ in \cite{MOSW} \\
10 &   5 & \cite{MOSW}& 4 & 2 & $E_{10},B_{10}$ in \cite{MOSW} \\
12 &  10 & \cite{MOSW}& 4 & 5
     & $E_{12},C_{12},D_{12},F_{12},2E_6$ in \cite{MOSW} \\
14 &  21 & \cite{MOSW}& 6 & 1 & \cite{MOSW} \\
16 &  55 & \cite{CPS} & 6 & 4 & \cite{CPS}  \\
18 & 245 & Section \ref{Sec:C} & 8 & 1 & \cite{Huffman} \\
20 &   ? & & 8 & 2 & \cite{Huffman} \\
\hline
   \end{tabular}
}
\end{center}
\end{table}

We list in  Table \ref{Tab:Ex} eleven Hermitian self-dual codes
$C_{10}$, $C_{14}$, $C_{15}$, $C_{17}$,
$C_{30}$, $C_{38}$, $C_{40}$, $C_{83}$,
$C_{120}$, $C_{147}$ and $C_{190}$ of minimum weight at least 4,
which are used in the next subsection.
Table \ref{Tab:Ex} lists
the dimension $\dim$ of $S_{18} \cap C_i$,
vectors $v_1,\ldots,v_{9-\dim}$ such that 
\[
C_i= \langle S_{18} \cap 
\langle v_1,\ldots,v_{9-\dim} \rangle^\perp,
v_1,\ldots,v_{9-\dim}\rangle,
\]
the numbers $A_4$ and $A_6$ of codewords of weights $4$ and $6$,
and the order $\#\Aut$ of the automorphism group of $C_i$.
By  \cite[Theorem 13]{MOSW}, the weight enumerator of 
a Hermitian self-dual code of length $18$
and minimum weight at least 4
 can be written as
\begin{align*}
&
1 + A_4 y^4 + A_6 y^6 
+ (2754   + 27 A_4  -  6 A_6) y^8 
+ (18360  - 106 A_4 + 15 A_6) y^{10} 
\\ &
+ (77112  + 119 A_4 - 20 A_6) y^{12} 
+ (110160 - 12 A_4  + 15 A_6) y^{14} 
\\ &
+ (50949  - 51 A_4  -  6 A_6) y^{16} 
+ (2808   + 22 A_4  +    A_6) y^{18}.
\end{align*}
Thus, 
the weight enumerator is uniquely determined by $A_4$ and $A_6$. 

\begin{table}[thb]
\caption{The codes $C_i$ ($i=10,14,15,17,30, 38, 40, 83, 120, 147, 190$)}
\label{Tab:Ex}
\begin{center}
{\footnotesize
\begin{tabular}{|c|c|l|rr|r|}
\hline
$i$ & $\dim$ &\multicolumn{1}{c|}{$v_1,\ldots,v_{9-\dim}$} 
&\multicolumn{1}{c}{$A_4$} 
&\multicolumn{1}{c|}{$A_6$} 
&\multicolumn{1}{c|}{$\#\Aut$} \\
\hline
{10} &8& $(1,1,1,\ww,1,1,1,1,1,\vv,\vv,\vv,\ww,\vv,\vv,\vv,0,0)$
&  0& 45&        180\\ \hline
{14} &8& $(1,1,1,1,1,1,1,1,1,1,\vv,\ww,0,0,0,0,0,0            )$
&  0& 27&       2916\\ \hline
{15} &8& $(1,\ww,1,1,1,1,1,1,1,1,\ww,\vv,0,\vv,0,1,\ww,\ww    )$
&  0& 27&        648\\ \hline
{17} &8& $(1,1,1,1,1,1,1,1,1,0,\ww,0,\ww,\vv,\ww,\ww,\ww,\ww  )$
&  0& 99&       1080\\ \hline
{30} &8& $(1,1,1,1,1,1,1,1,1,\ww,\ww,0,\vv,0,0,0,0,0          )$
&  9& 36&       2304\\ \hline
{38} &7& $(1,1,1,1,1,1,1,1,1,0,0,\vv,\ww,\vv,\ww,\ww,\ww,\ww  )$
&  0&108&      23328\\
     & & $(1,\ww,1,1,1,1,1,1,1,0,0,1,\vv,0,\vv,0,\vv,\vv      )$&&&\\ \hline
{40} &7& $(1,1,1,1,1,1,1,1,1,1,\ww,\ww,1,1,\ww,1,1,1          )$
&  0& 72&        216\\
     & & $(\ww,1,1,1,1,1,1,1,1,\vv,\ww,\ww,\vv,0,\ww,0,\ww,1  )$&&&\\ \hline
{83} &7& $(1,1,1,1,1,1,1,1,1,0,0,0,\vv,1,\ww,0,0,0            )$
&  9& 72&      62208\\
     & & $(\ww,1,1,1,1,1,1,1,1,0,\vv,\vv,\ww,0,\ww,\vv,1,\ww  )$&&&\\ \hline
{120}&7& $(1,1,1,1,1,1,1,1,1,1,0,0,\ww,\ww,\ww,1,1,1          )$
& 27& 18&     248832\\
     & & $(\ww,1,1,1,1,1,1,1,1,1,1,\ww,0,0,1,\vv,1,\ww        )$&&&\\ \hline
{147}&7& $(1,1,1,1,1,1,1,1,1,1,\vv,\vv,0,\ww,\vv,\ww,1,0      )$
& 45& 90&  $2^73^65^3$\\ 
     & & $(\ww,1,1,1,1,1,1,1,1,\vv,\vv,1,\vv,0,1,1,\vv,0      )$&&&\\ \hline
{190}&6& $(1,1,1,1,1,1,1,1,1,0,\vv,0,0,\vv,\vv,0,0,0          )$
&135& 54& $2^{10}3^{10}5^3$\\
     & & $(\ww,1,1,1,1,1,1,1,1,0,\vv,\vv,\vv,\vv,1,0,0,0      )$&&&\\
     & & $(1,\ww,1,1,1,1,1,1,1,0,1,0,0,\vv,\vv,1,0,\vv        )$&&&\\
\hline
   \end{tabular}
}
\end{center}
\end{table}

In the above classification,
we employed monomial matrices over $GF(4)$ in 
the definition for equivalence of codes.
To define a weaker equivalence,
one could consider a conjugation $\gamma$ of $GF(4)$ sending 
each element to its square in the  definition of equivalence,
that is,
two codes $C$ and $C'$ are {\em weakly} equivalent if there is a monomial
matrix $M$ over $GF(4)$ such that $C' =C M$
or $C'=CM\gamma$ (see \cite{Huffman}).

We have verified that
the equivalence classes 
of self-dual codes of lengths up to 16
are the same under both definitions.
For length $18$, there are $230$ classes under the weaker equivalence.
More specifically, the following codes are weakly equivalent:
\begin{align*}
&
(C_{  8}, C_{  9}),
(C_{ 10}, C_{ 11}),
(C_{ 19}, C_{ 20}),
(C_{ 24}, C_{ 25}),
(C_{ 26}, C_{ 27}),
\\&
(C_{ 28}, C_{ 29}),
(C_{ 30}, C_{ 31}),
(C_{ 50}, C_{ 51}),
(C_{ 56}, C_{ 57}),
(C_{ 73}, C_{ 74}),
\\&
(C_{ 89}, C_{ 90}),
(C_{ 92}, C_{ 93}),
(C_{ 94}, C_{ 95}),
(C_{113}, C_{114}),
(C_{118}, C_{119}).
\end{align*}

\section{A classification of generalized Hadamard matrices $H(6,3)$
based on codes}
\label{Sec:GH18}

Let $G=\langle \ww \rangle$ be the cyclic group of order 3 
being the multiplicative group of $GF(4)$.
Assume that $H(6,3)$ is a generalized Hadamard matrix
of order $18$ over $G$.
By Theorem \ref{thm:sd},
the code $C(H(6,3))$  generated by the rows of $H(6,3)$
is a Hermitian self-dual code over $GF(4)$
of length $18$ and minimum weight at least $4$.

Let $C$ be a Hermitian self-dual code of length $18$
and  minimum weight at least $4$.
We  define a simple undirected graph $\Gamma(C)$, 
whose set $V$ of vertices is the set of codewords
$x=(1,x_2,\ldots,x_{18})$ of weight $18$ in $C$,
with two vertices $x,y \in V$ being adjacent 
if $(n_1,n_{\ww},n_{\vv})=(6,6,6)$, where
$n_\alpha=\#\{i \mid x_i y_i^2=\alpha\}$.

The following statement was obtained by computations using {\sc Magma}.
\begin{lem}
Let $C$ be a Hermitian self-dual code of length $18$.
The graph $\Gamma(C)$ has a $18$-clique if and only if
$C$ is equivalent to one of the $13$ codes $C_{i}$
$(i=10,11,14,15,17,30,31,38,40,83,120,147,190)$.
\end{lem}

Note that
the eleven codes other than $C_{11}$, $C_{31}$ can be
found in Table \ref{Tab:Ex}, while 
the codes $C_{11}$ and $C_{31}$ are obtained as
$C_{10}\gamma$ and $C_{30}\gamma$, respectively.

The $18$-cliques in the graph $\Gamma(C)$
are 
generalized Hadamard matrices $H(6,3)$.
It is clear that $\Aut(C)$ acts on the graph $\Gamma(C)$
as a (not necessarily full) group of automorphisms.
If two $18$-cliques in $\Gamma(C)$ are in the
same $\Aut(C)$-orbit of the set of $18$-cliques in $\Gamma(C)$,
then the two generalized Hadamard matrices corresponding to
the two $18$-cliques are equivalent.
Hence, we found generalized Hadamard matrices corresponding to
representatives of $18$-cliques in $\Gamma(C)$ up to the action
of $\Aut(C)$. Then we verified whether two 
generalized Hadamard matrices are equivalent 
by a method  similar to that given in Section \ref{nets}.
For each code $C_i$, we list in Table \ref{Tab:GH}
the number $\#$ of generalized 
Hadamard matrices $H(6,3)$ which are not monomially equivalent,
obtained in this way.
In Table \ref{Tab:GH},
we also list corresponding
generalized Hadamard matrices given in Section \ref{nets}.
Therefore, we have an alternative classification of 
the generalized Hadamard matrices $H(6,3)$, 
given in Corollary \ref{Cor:H63}.

\begin{table}[thb]
\caption{Generalized Hadamard matrices in $C_i$}
\label{Tab:GH}
\begin{center}
{\small
\begin{tabular}{|c|c|l|}
\hline
$i$ & $\#$ & \multicolumn{1}{c|}{generalized  Hadamard matrices}  \\
\hline

 10 &  1 & $\overline{H(\cD_{45})}$ \\
 11 &  1 & $H(\cD_{45})$ \\
 14 &  4 & $H(\cD_{i}), \overline{H(\cD_{ i})}$ $(i=44, 53)$ \\
 15 &  3 & $H(\cD_{19})$, $H(\cD_{21}), \overline{H(\cD_{21})}$ \\
 17 &  8 & $H(\cD_{23})$, $H(\cD_{27})$, 
           $H(\cD_{i}), \overline{H(\cD_{ i})}$ $(i=24, 25, 26)$ \\
 30 &  2 & $H(\cD_{32})$, $\overline{H(\cD_{46})}$ \\
 31 &  2 & $H(\cD_{46})$, $\overline{H(\cD_{32})}$ \\
 38 &  9 & $H(\cD_{i})$ $(i=37,38,39)$, 
           $H(\cD_{j}), \overline{H(\cD_{ j})}$ $(j=34, 35, 36)$ \\
 40 &  3 & $H(\cD_{20})$, $H(\cD_{22}), \overline{H(\cD_{22})}$ \\
 83 & 12 & $H(\cD_{i})$ $(i=28, 33, 42, 43)$, 
           $H(\cD_{j}), \overline{H(\cD_{ j})}$ $(j=30, 31, 51, 52)$ \\
120 &  9 & $H(\cD_{i})$ $(i=1,2,3)$, 
           $H(\cD_{j}), \overline{H(\cD_{ j})}$ $(j=15, 16, 17)$ \\
147 & 14 & 
   $H(\cD_{i}), \overline{H(\cD_{ i})}$ $(i=29, 40, 41, 47, 48, 49, 50)$ \\
190 & 17 & $H(\cD_i)$ $(i=4, 5, 6, 12, 13, 14, 18)$, 
           $H(\cD_j), \overline{H(\cD_{ j})}$ $(j=7, 8, 9, 10, 11)$ \\
\hline
   \end{tabular}
}
\end{center}
\end{table}

\section{Acknowledgments}
The fourth co-author, Vladimir Tonchev, would like to thank 
Tohoku University, and Yamagata University for the hospitality
during his visit in June 2009.
The research of this co-author was partially supported by
NSA Grant H98230-10-1-0177.


\end{document}